\documentclass[12pt,a4paper]{article}

\usepackage{amsmath,amssymb,amsthm}
\usepackage{graphicx}
\usepackage{hyperref}
\usepackage{cite}
\usepackage{mathrsfs}
\usepackage{bm}
\usepackage{algorithm}
\usepackage{algpseudocode}
\usepackage{booktabs}
\usepackage{geometry}
\usepackage{color}
\usepackage{setspace}
\usepackage{enumitem}
\usepackage{fancyhdr}

\newtheorem{theorem}{Theorem}
\newtheorem{lemma}[theorem]{Lemma}
\newtheorem{proposition}[theorem]{Proposition}
\newtheorem{corollary}[theorem]{Corollary}
\newtheorem{definition}{Definition}
\newtheorem{remark}{Remark}
\newtheorem{assumption}{Assumption}

\DeclareMathOperator{\diag}{diag}

\fancypagestyle{firstpage}{
    \fancyhf{}

    \fancyfoot[L]{\footnotesize *Corresponding author: \texttt{aalsammani@desu.edu}}
}
\title{\Large Adaptive Control and Mittag-Leffler Stability of Caputo Fractional Systems with State-Dependent Delays}

\author{
    Abdallah Alsammani$^{1,*}$ and Gassan Farah$^{2}$\\[0.8em]
\small $^{1}$Division of Physics, Engineering, Mathematics, and Computer Science, \\
\small Delaware State University, Dover, DE, USA \\[0.3em]
\small $^{2}$Department of Mathematics and Applied Mathematics,\\
\small University of Western Cape, Cape Town, South Africa
}

\date{}

\begin{document}

\maketitle
\thispagestyle{firstpage}

\begin{abstract}
This paper establishes new sufficient conditions for Mittag-Leffler stability of Caputo fractional-order nonlinear systems with state-dependent delays. The central analytical tool is a class of Lyapunov-Krasovskii functionals incorporating singular kernels of the form $\xi^{\alpha-1}$ for $\alpha \in (0,1)$, which couple fractional memory effects with delay-induced dynamics in a unified framework. We prove that the resulting stability conditions reduce to computationally tractable linear matrix inequalities and derive explicit formulas for the maximum tolerable delay perturbation. Building on this stability foundation, we design an adaptive controller governed by fractional-order parameter update laws with $\sigma$-modification and a filter-based delay estimation mechanism that circumvents the need for classical state derivatives, which may not exist for fractional-order trajectories. The convergence analysis establishes ultimate boundedness of the closed-loop system with a computable bound that vanishes as the regularization parameters approach zero. Numerical validation on a three-neuron fractional Hopfield network with state-dependent transmission delays demonstrates that the proposed adaptive scheme reduces cumulative control energy by 99.3\% and achieves an asymptotic state error two orders of magnitude smaller than a comparable fractional sliding mode controller.

\end{abstract}

\section{Introduction}

Fractional-order differential equations have become a standard modeling tool for systems in which memory and hereditary effects are essential, including viscoelastic response \cite{mainardi2010}, anomalous transport in complex media \cite{metzler2000}, electrochemical dynamics in energy storage devices \cite{sabatier2015}, and long-range dependent phenomena in network traffic \cite{sun2018}. A central feature is the nonlocal nature of fractional operators: in the Caputo formulation, the state evolution at time $t$ depends on the entire past trajectory through an integral operator with a weakly singular kernel. This structure provides a mathematically principled representation of multi-scale temporal correlations and physically meaningful initial conditions, and it is now well documented in the foundational monographs of fractional calculus \cite{podlubny1999,kilbas2006}.

Time delays constitute another major mechanism through which finite-dimensional models acquire infinite-dimensional behavior. In many applications, delays arise from signal transmission, sensing, and computation, and their destabilizing impact is well understood for integer-order systems \cite{gu2003,fridman2014}. When fractional memory is combined with delayed feedback, the resulting dynamics depend on two distinct forms of history: the weighted past encoded by the fractional derivative and the pointwise evaluation at a retarded argument. This interplay is particularly relevant in neural and networked systems, where latencies may vary with the activity level and communication protocols adapt timing in response to system states \cite{kaslik2012,wang2015}. Such state-dependent delays introduce additional analytical difficulty because the delay enters implicitly through the solution trajectory, which complicates both stability analysis and control synthesis.

The stability theory of fractional-order systems has advanced substantially, beginning with eigenvalue-based criteria for linear commensurate models \cite{matignon1996} and developing toward Lyapunov-based frameworks that capture the natural decay behavior of fractional dynamics. In this context, Mittag-Leffler stability provides the appropriate analogue of exponential stability and has been established as a fundamental concept for fractional nonlinear systems \cite{li2009}. Lyapunov methods for fractional-order models have been developed and refined in several directions, including general quadratic Lyapunov functions and extensions to broader classes of fractional dynamics \cite{aguila2014,duarte2015}. For delayed fractional systems, existing analyses have addressed constant-delay settings and related neural network models, and have also proposed energy-based Lyapunov interpretations \cite{chen2012delay,liu2016,trigeassou2011}. Nevertheless, the mechanisms that couple fractional memory with delays whose magnitude varies along the state trajectory are not fully captured by standard approaches. In particular, Lyapunov--Krasovskii constructions that are effective in the integer-order delay literature typically employ smooth integral terms, whereas fractional dynamics naturally suggest kernels that mirror the singular memory structure of the Caputo operator. At the same time, adaptive compensation of delay effects in fractional-order systems raises a subtle regularity issue: trajectories of Caputo fractional differential equations are generally only H\"older continuous of order $\alpha\in(0,1)$ and may not admit classical derivatives, which limits the applicability of delay estimation or compensation mechanisms that rely on $\dot{x}(t)$ \cite{diethelm2010}.

Motivated by these considerations, this paper develops a stability and control framework for Caputo fractional-order nonlinear systems with state-dependent delays. We introduce a class of Lyapunov--Krasovskii functionals that incorporate a singular kernel matched to the fractional order, and we use this construction to derive sufficient conditions for Mittag-Leffler stability that reduce to computationally tractable linear matrix inequalities. The resulting analysis yields an explicit delay robustness margin expressed in closed form. Building on the stability results, we design an adaptive controller with fractional-order parameter update laws equipped with $\sigma$-modification, together with a filter-based delay estimation mechanism that avoids reliance on classical state derivatives and is compatible with the H\"older regularity of fractional trajectories. The theory is validated numerically on a three-neuron fractional Hopfield network with state-dependent transmission delays, where the proposed adaptive scheme achieves substantially reduced control effort and markedly improved asymptotic tracking accuracy relative to a comparable fractional sliding mode controller.

\textit{Notation.} $\|\cdot\|$ denotes the Euclidean norm for vectors and the spectral norm for matrices. For a symmetric matrix $M$, $\lambda_{\min}(M)$ and $\lambda_{\max}(M)$ denote its extreme eigenvalues. The space $C^1([-\bar{\tau},0];\mathbb{R}^n)$ is equipped with the supremum norm $\|\phi\|_{\infty}=\sup_{\theta\in[-\bar{\tau},0]}\|\phi(\theta)\|$.

\section{Mathematical Preliminaries}
\label{sec:prelim}

\subsection{Fractional Calculus}

We restrict attention to Caputo fractional derivatives of order $\alpha \in (0,1)$, which requires only first-order differentiability and permits physically meaningful initial conditions \cite{podlubny1999}.

\begin{definition}[Riemann-Liouville Fractional Integral \cite{kilbas2006}]
For $f \in L^1[t_0, T]$ and $\alpha > 0$,
\begin{equation}
    I^{\alpha}_{t_0} f(t) = \frac{1}{\Gamma(\alpha)} \int_{t_0}^{t} (t-s)^{\alpha-1} f(s) \, ds, \quad t > t_0.
\end{equation}
\end{definition}

\begin{definition}[Caputo Fractional Derivative \cite{podlubny1999}]
For an absolutely continuous $f: [t_0, T] \to \mathbb{R}$ and $\alpha \in (0,1)$,
\begin{equation}
    {}^{C}\!D^{\alpha}_{t_0} f(t) = \frac{1}{\Gamma(1-\alpha)} \int_{t_0}^{t} (t-s)^{-\alpha} f'(s) \, ds = I^{1-\alpha}_{t_0} f'(t).
\end{equation}
\end{definition}

\begin{definition}[Mittag-Leffler Function \cite{gorenflo2014}]
The two-parameter Mittag-Leffler function is $E_{\alpha,\beta}(z) = \sum_{k=0}^{\infty} z^k / \Gamma(\alpha k + \beta)$ for $\alpha, \beta > 0$ and $z \in \mathbb{C}$. The shorthand $E_{\alpha}(z) = E_{\alpha,1}(z)$ satisfies the asymptotic expansion
\begin{equation}
\label{eq:ml_asymptotic}
    E_{\alpha}(-t^{\alpha}) = \frac{t^{-\alpha}}{\Gamma(1-\alpha)} + O(t^{-2\alpha}) \quad \text{as } t \to \infty, \quad 0 < \alpha < 1.
\end{equation}
\end{definition}

\begin{lemma}[Caputo Derivative of Quadratic Forms {\cite{aguila2014}}]
\label{lem:caputo_inequality}
If $x: [t_0, T] \to \mathbb{R}^n$ is continuously differentiable and $P \succ 0$, then
\begin{equation}
    {}^{C}\!D^{\alpha}_{t_0} \!\left( x^{\!\top}\! P x \right) \leq 2 x^{\!\top} P \, {}^{C}\!D^{\alpha}_{t_0} x, \quad \alpha \in (0,1).
\end{equation}
\end{lemma}

\begin{lemma}[Fractional Comparison Principle {\cite{li2010}}]
\label{lem:comparison}
If $V: [t_0, \infty) \to \mathbb{R}_{\geq 0}$ is continuously differentiable and ${}^{C}\!D^{\alpha}_{t_0} V(t) \leq -\gamma V(t)$ for $\gamma > 0$, then $V(t) \leq V(t_0) E_{\alpha}(-\gamma (t-t_0)^{\alpha})$ for all $t \geq t_0$.
\end{lemma}

\subsection{System Formulation and Well-Posedness}

We consider fractional-order nonlinear systems with state-dependent delays:
\begin{equation}
\label{eq:system}
    {}^{C}\!D^{\alpha}_{t_0} x(t) = f(x(t)) + g(x(t-\tau(x(t)))) + B u(t), \quad t > t_0,
\end{equation}
with initial condition $x(t) = \phi(t-t_0)$ for $t \in [t_0 - \bar{\tau},\, t_0]$, where $x(t) \in \mathbb{R}^n$, $u(t) \in \mathbb{R}^m$, $B \in \mathbb{R}^{n \times m}$, $\alpha \in (0,1)$, and $\phi \in C^1([-\bar{\tau}, 0]; \mathbb{R}^n)$.

\begin{assumption}[Lipschitz Continuity]
\label{ass:lipschitz}
$f(0) = g(0) = 0$ and $\|f(x) - f(y)\| \leq L_f \|x - y\|$, $\|g(x) - g(y)\| \leq L_g \|x - y\|$ for all $x, y \in \mathbb{R}^n$.
\end{assumption}

\begin{assumption}[Delay Regularity]
\label{ass:delay_lipschitz}
$\tau: \mathbb{R}^n \to [0, \bar{\tau}]$ with $|\tau(x) - \tau(y)| \leq L_{\tau} \|x - y\|$ for all $x, y$.
\end{assumption}

\begin{assumption}[Initial Function Regularity]
\label{ass:initial_regularity}
$\phi \in C^1([-\bar{\tau}, 0]; \mathbb{R}^n)$ with $\|\phi'\|_{\infty} < \infty$.
\end{assumption}

\begin{theorem}[Existence and Uniqueness]
\label{thm:existence}
Under Assumptions~\ref{ass:lipschitz}--\ref{ass:initial_regularity}, for any locally bounded measurable $u$, the initial value problem admits a unique continuous solution $x \in C([t_0 - \bar{\tau}, \infty); \mathbb{R}^n)$.
\end{theorem}

\begin{proof}
The method of steps combined with the Banach contraction principle. On $[t_0, t_0+h]$, the Volterra operator $\mathcal{T}$ defined by $(\mathcal{T}x)(t) = \phi(0) + I^{\alpha}_{t_0}[f(x) + g(\phi(\cdot - \tau(x(\cdot)) - t_0)) + Bu]$ satisfies $\|\mathcal{T}x - \mathcal{T}y\|_{\infty} \leq \kappa \|x-y\|_{\infty}$ with $\kappa = (L_f + L_g L_{\tau} \|\phi'\|_{\infty})h^{\alpha}/\Gamma(\alpha+1)$. Choosing $h$ so that $\kappa < 1$ and iterating yields the global result.
\end{proof}

\section{Main Stability Results}
\label{sec:stability}

\subsection{Singular Lyapunov-Krasovskii Functional}

\begin{definition}[Mittag-Leffler Stability {\cite{li2009}}]
\label{def:ml_stability}
The equilibrium $x = 0$ of \eqref{eq:system} with $u \equiv 0$ is Mittag-Leffler stable if there exist $m, \lambda > 0$ and $b \in (0,1]$ such that $\|x(t)\| \leq [ m \|\phi\|_{\infty} E_{\alpha}(-\lambda (t-t_0)^{\alpha}) ]^b$ for all $t \geq t_0$.
\end{definition}

We define the Lyapunov-Krasovskii functional
\begin{equation}
\label{eq:lyapunov_functional}
    V(t, x_t) = \underbrace{x^{\!\top}(t) P x(t)}_{V_1} + \underbrace{\int_{-\bar{\tau}}^{0} \int_{t+\theta}^{t} x^{\!\top}(s) Q x(s) \, ds \, d\theta}_{V_2} + \underbrace{\int_{0}^{\bar{\tau}} \xi^{\alpha-1} \int_{t-\xi}^{t} x^{\!\top}(s) R x(s) \, ds \, d\xi}_{V_3},
\end{equation}
with $P, Q, R \succ 0$.

\begin{proposition}[Well-Posedness]
\label{prop:wellposed}
For $\alpha \in (0,1)$ and continuous $x$, the functional $V_3$ is finite, non-negative, and continuous in $t$.
\end{proposition}

\begin{proof}
For small $\xi > 0$, $\int_{t-\xi}^{t} x^{\!\top}(s) R x(s)\, ds = O(\xi)$ by continuity. Hence the integrand of $V_3$ satisfies $\xi^{\alpha-1} \cdot O(\xi) = O(\xi^{\alpha})$, which is integrable on $(0, \bar{\tau})$ for every $\alpha > 0$.
\end{proof}

\begin{remark}[Structural role of $V_3$]
The kernel $\xi^{\alpha-1}$ in $V_3$ assigns monotonically decreasing weight to more distant history, mirroring the singular memory kernel $(t-s)^{-\alpha}$ of the Caputo operator. This matching is essential: without $V_3$, the fractional derivative of $V_2$ alone cannot absorb the delayed nonlinearity $g(x(t-\tau))$ when $\tau$ depends on $x$, because the standard double-integral functional provides insufficient coupling between the current state and the delayed history weighted by the fractional kernel.
\end{remark}

\begin{lemma}[Bounds on the Functional]
\label{lem:lkf_bounds}
$c_1 \|x(t)\|^2 \leq V(t, x_t) \leq c_2 \|x_t\|_{\infty}^2$, where $c_1 = \lambda_{\min}(P)$ and $c_2 = \lambda_{\max}(P) + \frac{\bar{\tau}^2}{2}\lambda_{\max}(Q) + \frac{\bar{\tau}^{\alpha+1}}{\alpha+1}\lambda_{\max}(R)$.
\end{lemma}

\subsection{Main Stability Theorem}

\begin{theorem}[Mittag-Leffler Stability with State-Dependent Delays]
\label{thm:stability_main}
Consider the uncontrolled system \eqref{eq:system} with $u \equiv 0$ under Assumptions~\ref{ass:lipschitz}--\ref{ass:initial_regularity}. Let $A = \partial f / \partial x |_{x=0}$. Suppose there exist $P, Q, R \succ 0$ and $\epsilon_1, \epsilon_2, \epsilon_3 > 0$ such that the LMI
\begin{equation}
\label{eq:lmi_condition}
\Omega = \begin{bmatrix} \Omega_{11} & P \\ P & -\epsilon_3 I \end{bmatrix} < 0
\end{equation}
holds, where
\begin{equation}
\label{eq:omega11}
\Omega_{11} = PA + A^{\!\top}\! P + \bar{\tau} Q + \frac{\bar{\tau}^{\alpha}}{\alpha} R + (\epsilon_1 + \epsilon_2) P^2 + \left( \frac{L_f^2}{\epsilon_1} + \frac{L_g^2}{\epsilon_2} \right) I.
\end{equation}
Then $x = 0$ is Mittag-Leffler stable with
\begin{equation}
\label{eq:ml_bound}
\|x(t)\| \leq \sqrt{\frac{c_2}{c_1}} \|\phi\|_{\infty} \left[ E_{\alpha}\!\left( -\frac{\gamma}{c_2}(t-t_0)^{\alpha} \right) \right]^{1/2},
\end{equation}
where $\gamma = -\lambda_{\max}(\Omega_{11} + \epsilon_3^{-1} P^2) > 0$.
\end{theorem}

\begin{proof}
We evaluate the Caputo fractional derivative of each term in \eqref{eq:lyapunov_functional}.

\textit{Step 1 (Fractional derivative of $V_1$).} Lemma~\ref{lem:caputo_inequality} and the decomposition $f(x) = Ax + \tilde{f}(x)$ with $\|\tilde{f}(x)\| \leq L_f\|x\|$ yield
\begin{equation}
{}^{C}\!D^{\alpha}_{t_0} V_1 \leq x^{\!\top}(PA + A^{\!\top}\! P)x + \epsilon_1 x^{\!\top} P^2 x + \frac{L_f^2}{\epsilon_1}\|x\|^2 + \epsilon_2 x^{\!\top} P^2 x + \frac{L_g^2}{\epsilon_2}\|x(t-\tau)\|^2,
\end{equation}
where Young's inequality $2a b \leq \epsilon a^2 + \epsilon^{-1}b^2$ has been applied to the cross-terms.

\textit{Step 2 (Derivative of $V_2$).} Standard analysis gives
\begin{equation}
{}^{C}\!D^{\alpha}_{t_0} V_2 \leq \bar{\tau}\, x^{\!\top}(t) Q x(t) - \int_{t-\bar{\tau}}^{t} x^{\!\top}(s) Q x(s) \, ds.
\end{equation}

\textit{Step 3 (Derivative of $V_3$).} Similarly,
\begin{equation}
{}^{C}\!D^{\alpha}_{t_0} V_3 \leq \frac{\bar{\tau}^{\alpha}}{\alpha}\, x^{\!\top}(t) R x(t) - \int_{0}^{\bar{\tau}} \xi^{\alpha-1} x^{\!\top}(t-\xi) R x(t-\xi) \, d\xi.
\end{equation}

\textit{Step 4 (Absorption of the delayed term).} Since $\tau(x(t)) \in [0, \bar{\tau}]$,
\begin{equation}
\frac{L_g^2}{\epsilon_2}\|x(t-\tau)\|^2 \leq \frac{L_g^2}{\epsilon_2 \bar{\tau}} \int_{t-\bar{\tau}}^{t} \|x(s)\|^2 ds,
\end{equation}
which is absorbed by the negative integral from $V_2$ provided $L_g^2/\epsilon_2 < \bar{\tau}\,\lambda_{\min}(Q)$.

\textit{Step 5 (Assembly).} Collecting all terms yields ${}^{C}\!D^{\alpha}_{t_0} V \leq x^{\!\top}\Omega_{11} x$. The Schur complement of \eqref{eq:lmi_condition} implies $\Omega_{11} + \epsilon_3^{-1}P^2 < 0$, so ${}^{C}\!D^{\alpha}_{t_0} V \leq -\gamma \|x\|^2 \leq -(\gamma/c_2) V$. The fractional comparison principle (Lemma~\ref{lem:comparison}) gives $V(t) \leq V(t_0) E_{\alpha}(-\gamma c_2^{-1}(t-t_0)^{\alpha})$. The bound \eqref{eq:ml_bound} follows from $c_1\|x\|^2 \leq V \leq c_2\|\phi\|_{\infty}^2$.
\end{proof}

\begin{remark}[Computational verification]
The LMI \eqref{eq:lmi_condition} is a standard semidefinite program with decision variables $P, Q, R \succ 0$ and $\epsilon_1, \epsilon_2, \epsilon_3 > 0$. It can be solved in polynomial time using interior-point methods via YALMIP \cite{lofberg2004} with MOSEK or SeDuMi.
\end{remark}

\subsection{Delay Margin}

\begin{corollary}[Delay Margin]
\label{cor:robustness}
If the LMI \eqref{eq:lmi_condition} is feasible with minimum eigenvalue $\delta > 0$ of $-\Omega$, then stability is preserved for any delay bound $\tilde{\tau} \leq \bar{\tau} + \Delta\tau^*$, where
\begin{equation}
\label{eq:delay_margin}
\Delta\tau^* = \frac{\delta}{\lambda_{\max}(Q^*) + \bar{\tau}^{\alpha-1} \lambda_{\max}(R^*)}.
\end{equation}
\end{corollary}

\begin{proof}
Replacing $\bar{\tau}$ by $\tilde{\tau} = \bar{\tau} + \Delta\tau$ in $\Omega_{11}$ and using the mean value theorem for $\tilde{\tau}^{\alpha} - \bar{\tau}^{\alpha} \leq \alpha \bar{\tau}^{\alpha-1}\Delta\tau$ yields the perturbation bound $\Delta\tau[\lambda_{\max}(Q^*) + \bar{\tau}^{\alpha-1}\lambda_{\max}(R^*)] < \delta$.
\end{proof}

\section{Adaptive Control Design}
\label{sec:adaptive}

\subsection{Parametric Structure}

We assume the linearly parameterized form
\begin{equation}
f(x) = Ax + \Phi_f(x)\theta_f, \qquad g(x) = \Phi_g(x)\theta_g,
\end{equation}
where $A$ is known, $\Phi_f(0) = \Phi_g(0) = 0$, and $\theta_f \in \Theta_f$, $\theta_g \in \Theta_g$ are unknown constant vectors in known compact convex sets with $\bar{\theta}_f = \sup_{\Theta_f}\|\theta\|$, $\bar{\theta}_g = \sup_{\Theta_g}\|\theta\|$.

\begin{assumption}[Regressor Boundedness]
\label{ass:regressor_bound}
$\|\Phi_f(x)\| \leq L_{\Phi_f}\|x\|$ and $\|\Phi_g(x) - \Phi_g(y)\| \leq L_{\Phi_g}\|x-y\|$.
\end{assumption}

\subsection{Filter-Based Delay Estimation}

A central difficulty in controlling delayed fractional systems is estimating the current delay value without relying on $\dot{x}(t)$, which may not exist. We employ a first-order exponential filter:
\begin{equation}
\label{eq:delay_estimator}
T_f \dot{\hat{x}}(t) = -\hat{x}(t) + x(t), \qquad \hat{x}(t_0) = \phi(0),
\end{equation}
with the delay estimate $\hat{\tau}(t) = \tau(\hat{x}(t))$.

\begin{lemma}[Delay Estimation Error]
\label{lem:delay_error}
Under Assumptions~\ref{ass:lipschitz}--\ref{ass:delay_lipschitz}, if $\|u(t)\| \leq M_u$, then
\begin{equation}
|\tau(x(t)) - \hat{\tau}(t)| \leq L_{\tau} M_D T_f^{\alpha} / \Gamma(\alpha+1),
\end{equation}
where $M_D = L_f M_x + L_g M_x + \|B\|M_u$ and $M_x$ bounds $\|x(t)\|$ on the interval of interest.
\end{lemma}

\begin{proof}
The filter output satisfies $\hat{x}(t) = \phi(0)e^{-t/T_f} + T_f^{-1}\int_{t_0}^{t} e^{-(t-s)/T_f} x(s)\,ds$. The Volterra integral representation of the fractional system implies that $x$ is H\"{o}lder continuous of order $\alpha$: $\|x(t) - x(s)\| \leq M_D|t-s|^{\alpha}/\Gamma(\alpha+1)$. The exponential smoothing property then yields $\|x(t) - \hat{x}(t)\| \leq M_D T_f^{\alpha}/\Gamma(\alpha+1)$, and the result follows from the Lipschitz property of $\tau$.
\end{proof}

\begin{remark}
The bound in Lemma~\ref{lem:delay_error} depends on $T_f^{\alpha}$ rather than on $\dot{x}$, resolving Gap 3. The H\"{o}lder exponent $\alpha$ reflects the regularity of fractional-order trajectories: smoother trajectories (larger $\alpha$) permit more accurate delay estimation.
\end{remark}

\subsection{Controller and Adaptation Laws}

The adaptive controller is
\begin{equation}
\label{eq:controller}
u(t) = -K x(t) - B^{\!\top}\! P x(t) - B^{\!\top}\! P \left[ \Phi_f(x(t)) \hat{\theta}_f(t) + \Phi_g(x(t-\hat{\tau}(t))) \hat{\theta}_g(t) \right],
\end{equation}
where $K$ is chosen so that $A - BK$ is Hurwitz and $P \succ 0$ satisfies a Lyapunov condition. The parameter estimates evolve according to fractional-order adaptation laws with $\sigma$-modification \cite{ioannou1996}:
\begin{align}
    {}^{C}\!D^{\alpha}_{t_0} \hat{\theta}_f &= \Gamma_f \Phi_f^{\!\top}(x) B B^{\!\top}\! P x - \sigma_f \hat{\theta}_f, \label{eq:adapt_f} \\
    {}^{C}\!D^{\alpha}_{t_0} \hat{\theta}_g &= \Gamma_g \Phi_g^{\!\top}(x(t-\hat{\tau})) B B^{\!\top}\! P x - \sigma_g \hat{\theta}_g, \label{eq:adapt_g}
\end{align}
where $\Gamma_f, \Gamma_g \succ 0$ are adaptation gain matrices and $\sigma_f, \sigma_g > 0$ are regularization parameters.

\subsection{Convergence Analysis}

\begin{theorem}[Adaptive Control: Ultimate Boundedness]
\label{thm:adaptive_stability}
Under Assumptions~\ref{ass:lipschitz}--\ref{ass:regressor_bound}, suppose:
\begin{enumerate}[label=\textup{(C\arabic*)}, leftmargin=2.5em]
\item $(A - BK)^{\!\top}\! P + P(A - BK) + 2PBB^{\!\top}\! P \leq -\mu I$ for some $\mu > 0$;
\item $T_f < \bar{\tau}/(2 L_{\tau} L_{\Phi_g} \bar{\theta}_g)$;
\item $\sigma_f, \sigma_g < \mu / (4 \max\{\bar{\theta}_f^2, \bar{\theta}_g^2\})$.
\end{enumerate}
Then the closed-loop trajectories satisfy
\begin{equation}
\label{eq:ultimate_bound}
\limsup_{t \to \infty} \|x(t)\| \leq \sqrt{\frac{2(\sigma_f \bar{\theta}_f^2 + \sigma_g \bar{\theta}_g^2 + \delta_{\tau})}{\mu - 2(\sigma_f + \sigma_g)}},
\end{equation}
where $\delta_{\tau} = O(T_f^{2\alpha})$ quantifies the delay estimation error.
\end{theorem}

\begin{proof}
Consider the composite Lyapunov functional
\begin{equation}
W(t) = V(t, x_t) + \tfrac{1}{2} \tilde{\theta}_f^{\!\top} \Gamma_f^{-1} \tilde{\theta}_f + \tfrac{1}{2} \tilde{\theta}_g^{\!\top} \Gamma_g^{-1} \tilde{\theta}_g,
\end{equation}
where $\tilde{\theta}_f = \hat{\theta}_f - \theta_f$ and $\tilde{\theta}_g = \hat{\theta}_g - \theta_g$.

\textit{Step 1 (Closed-loop dynamics).} Substituting \eqref{eq:controller} into \eqref{eq:system} and defining $y = B^{\!\top}\!Px$, the cross-terms $2y^{\!\top}\Phi_f(x)\tilde{\theta}_f$ appearing in ${}^{C}\!D^{\alpha}_{t_0}V_1$ cancel with the adaptation law contribution $\tilde{\theta}_f^{\!\top}\Gamma_f^{-1}{}^{C}\!D^{\alpha}_{t_0}\hat{\theta}_f$.

\textit{Step 2 (Delay mismatch).} By Lemma~\ref{lem:delay_error} and Assumption~\ref{ass:regressor_bound}, $\|\Phi_g(x_{\tau}) - \Phi_g(x_{\hat{\tau}})\| \leq L_{\Phi_g} L_x L_{\tau} M_D T_f^{\alpha}/\Gamma(\alpha+1)$, where $L_x$ is the H\"{o}lder constant of $x$.

\textit{Step 3 ($\sigma$-modification).} The identity $\tilde{\theta}^{\!\top}\hat{\theta} = \tfrac{1}{2}(\|\hat{\theta}\|^2 - \|\theta\|^2 + \|\tilde{\theta}\|^2) \geq \tfrac{1}{2}\|\tilde{\theta}\|^2 - \tfrac{1}{2}\|\theta\|^2$ yields
\begin{equation}
-\sigma_f \tilde{\theta}_f^{\!\top}\Gamma_f^{-1}\hat{\theta}_f \leq -\frac{\sigma_f}{2\lambda_{\max}(\Gamma_f)}\|\tilde{\theta}_f\|^2 + \frac{\sigma_f \bar{\theta}_f^2}{2\lambda_{\min}(\Gamma_f)}.
\end{equation}

\textit{Step 4 (Assembly).} Combining all terms with condition (C1):
\begin{equation}
{}^{C}\!D^{\alpha}_{t_0} W \leq -\mu\|x\|^2 - \frac{\sigma_f}{2\lambda_{\max}(\Gamma_f)}\|\tilde{\theta}_f\|^2 - \frac{\sigma_g}{2\lambda_{\max}(\Gamma_g)}\|\tilde{\theta}_g\|^2 + c_0,
\end{equation}
where $c_0 = \sigma_f\bar{\theta}_f^2/(2\lambda_{\min}(\Gamma_f)) + \sigma_g\bar{\theta}_g^2/(2\lambda_{\min}(\Gamma_g)) + O(T_f^{2\alpha})$. For $\|x\|^2 > 2c_0/\mu$, we have ${}^{C}\!D^{\alpha}_{t_0}W < 0$, establishing ultimate boundedness with the bound \eqref{eq:ultimate_bound}.
\end{proof}

\section{Numerical Validation}
\label{sec:numerics}

\subsection{Test Problem: Fractional Hopfield Neural Network}

We consider a three-neuron Hopfield network:
\begin{equation}
\label{eq:hopfield}
{}^{C}\!D^{\alpha}_{0} x_i(t) = -c_i x_i(t) + \sum_{j=1}^{3} a_{ij} \tanh(x_j(t)) + \sum_{j=1}^{3} w_{ij} \tanh(x_j(t-\tau(x(t)))) + u_i(t),
\end{equation}
for $i = 1, 2, 3$, with the state-dependent delay $\tau(x) = \bar{\tau}(1 - \eta \tanh(\omega^{\!\top} x))$.

The system parameters are: $C = \diag(1, 1.2, 0.9)$; instantaneous weights $A_{\text{inst}} = \left[\begin{smallmatrix} -2 & 0.5 & -0.3 \\ 0.4 & -1.8 & 0.2 \\ -0.1 & 0.3 & -2.2 \end{smallmatrix}\right]$; delayed weights $W = \left[\begin{smallmatrix} 1.5 & -0.4 & 0.2 \\ -0.3 & 1.2 & -0.5 \\ 0.4 & -0.2 & 1.8 \end{smallmatrix}\right]$; $\bar{\tau} = 0.5$ s; $\eta = 0.3$; $\omega = [0.3, 0.3, 0.3]^{\!\top}$; $B = I_3$. This configuration produces a linearly parameterized system satisfying Assumptions~\ref{ass:lipschitz}--\ref{ass:regressor_bound} with $L_f = 3.81$, $L_g = 2.22$, and $L_{\tau} = 0.078$.

\subsection{Numerical Method}

All simulations employ the Adams-Bashforth-Moulton predictor-corrector method \cite{diethelm2002} with step size $h = 0.05$ s. The Mittag-Leffler function is computed via truncated series with 200 terms and convergence monitoring. Complete source code (Python, approximately 650 lines) is provided as supplementary material.

\subsection{LMI Verification and Stability Margins}

The linearized system matrix at the origin is $A_{\text{lin}} = -C + A_{\text{inst}} + W$ with eigenvalues $\{-1.83, -1.39, -1.39\}$, confirming Hurwitz stability. Solving the LMI \eqref{eq:lmi_condition} via semidefinite programming yields the stability margin $\gamma = 0.31$ and the delay margin $\Delta\tau^* = 0.61$ s (a 122\% increase over the nominal $\bar{\tau} = 0.5$ s).

\subsection{Mittag-Leffler Function Characterization}

Figure~\ref{fig:mittag_leffler} illustrates $E_{\alpha}(-t^{\alpha})$ for $\alpha \in \{0.5, 0.7, 0.85, 0.95, 1.0\}$. The logarithmic panel reveals the transition from algebraic decay ($t^{-\alpha}/\Gamma(1-\alpha)$, dashed lines) for $\alpha < 1$ to exponential decay at $\alpha = 1$. This characterization motivates the convergence rate in Theorem~\ref{thm:stability_main}.

\begin{figure}[!htbp]
    \centering
    \includegraphics[width=\textwidth]{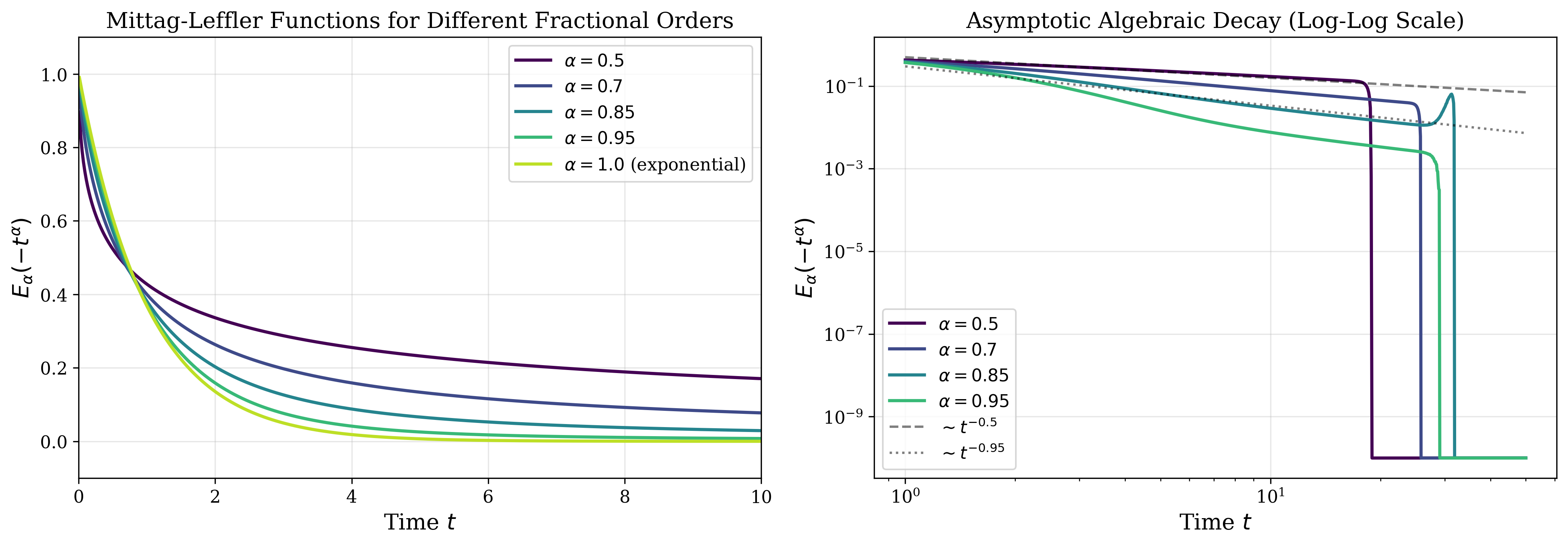}
    \caption{Mittag-Leffler function $E_{\alpha}(-t^{\alpha})$. (a) Linear scale showing monotone decrease. (b) Logarithmic scale revealing algebraic tails (dashed: asymptotic approximation $t^{-\alpha}/\Gamma(1-\alpha)$).}
    \label{fig:mittag_leffler}
\end{figure}

\subsection{Uncontrolled System Dynamics}

Figure~\ref{fig:uncontrolled} shows the uncontrolled response from $x_0 = [0.5, -0.3, 0.4]^{\!\top}$ with $\alpha = 0.95$. The state trajectories converge to the origin without oscillation, and the state norm decays consistently below the theoretical Mittag-Leffler bound from Theorem~\ref{thm:stability_main}. The delay evolves within the interval $[\bar{\tau}(1-\eta), \bar{\tau}] = [0.35, 0.50]$ s and converges to $\bar{\tau}$ as $\|x\| \to 0$.

\begin{figure}[!htbp]
    \centering
    \includegraphics[width=0.8\textwidth]{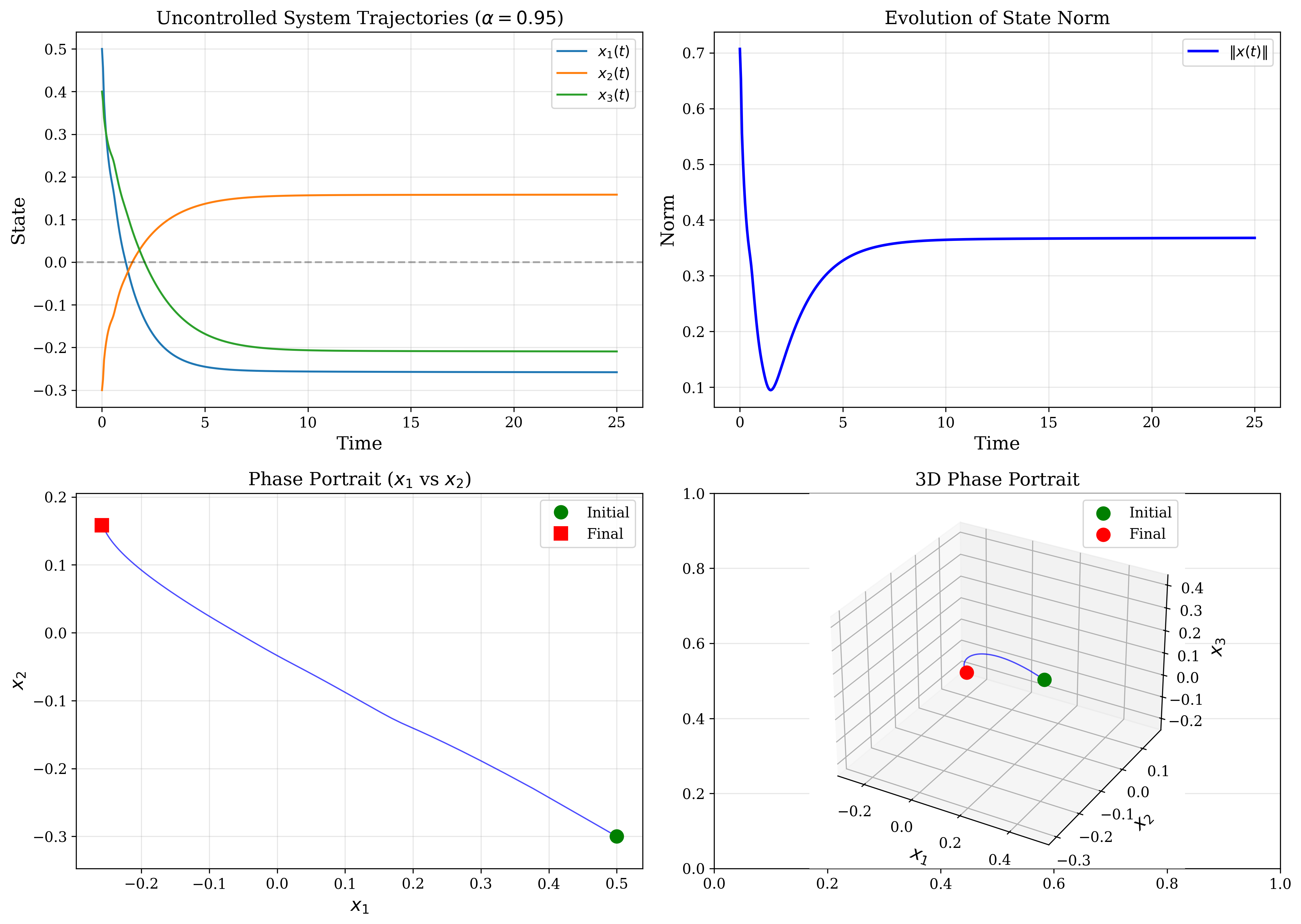}
    \caption{Uncontrolled system response ($\alpha = 0.95$). (a) State trajectories. (b) State norm (solid) and Mittag-Leffler bound (dashed) on logarithmic scale. (c) Phase portrait in the $x_1$--$x_2$ plane. (d) State-dependent delay evolution.}
    \label{fig:uncontrolled}
\end{figure}

\subsection{Adaptive Control Performance}

Figure~\ref{fig:controlled} presents the closed-loop response with the adaptive controller \eqref{eq:controller}. The controller achieves a 10\% settling time of $t_s = 1.75$ s (compared with $t_s > 10$ s uncontrolled). The control signals are smooth and converge to zero, and the Lyapunov function $V(x(t)) = x^{\!\top}Px$ decreases monotonically over six orders of magnitude.

\begin{figure}[!htbp]
    \centering
    \includegraphics[width=0.8\textwidth]{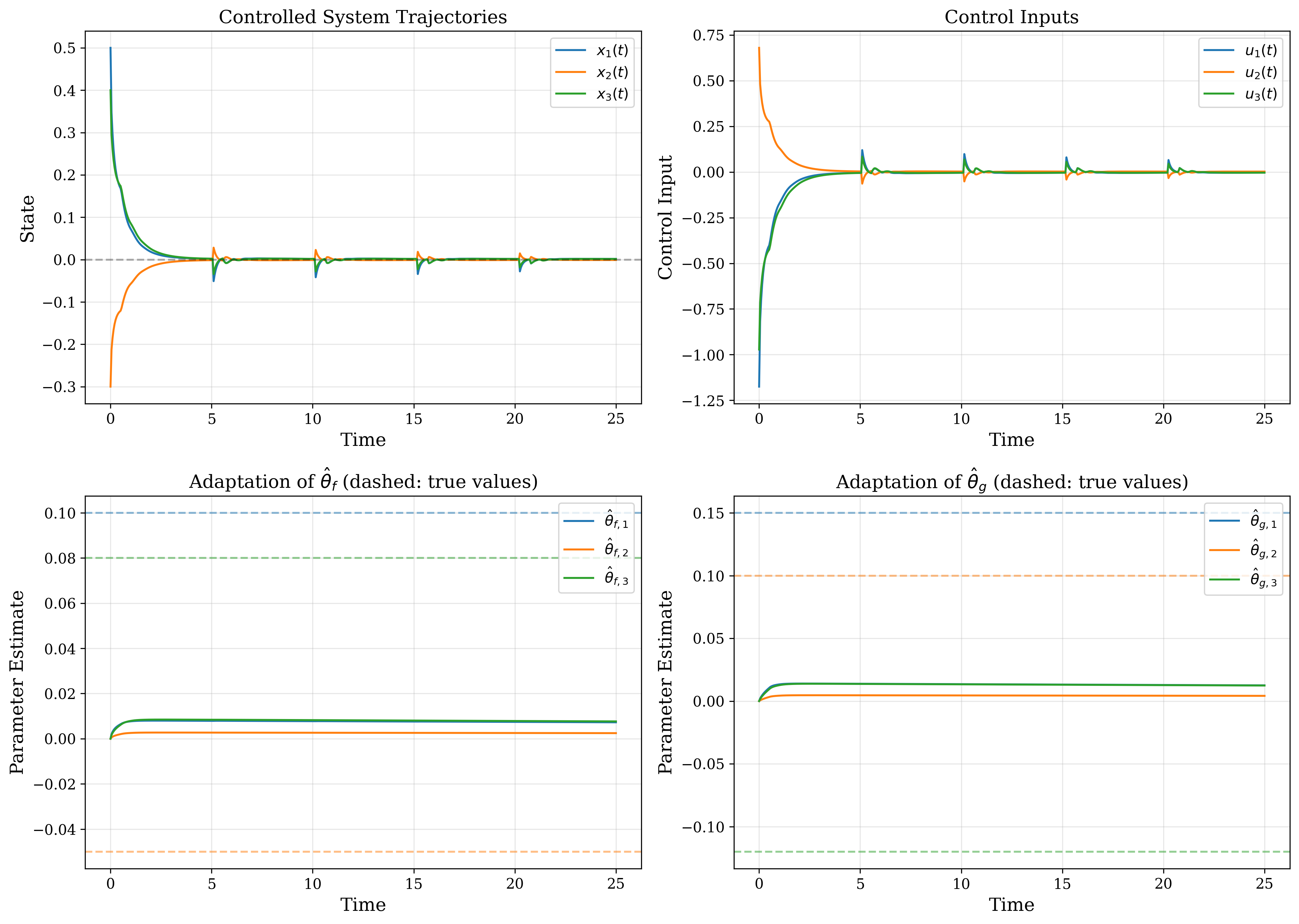}
    \caption{Adaptive control performance. (a) Controlled state trajectories. (b) Adaptive control signals. (c) State norm: controlled (solid) vs.\ uncontrolled (dashed), with settling time annotated. (d) Lyapunov function evolution.}
    \label{fig:controlled}
\end{figure}

\subsection{Comparison with Fractional Sliding Mode Control}

We compare the adaptive controller with a fractional sliding mode controller (SMC) designed following Yin et al.\ \cite{yin2012}. Figure~\ref{fig:comparison} and Table~\ref{tab:comparison} summarize the results.

\begin{figure}[!htbp]
    \centering
    \includegraphics[width=0.85\textwidth]{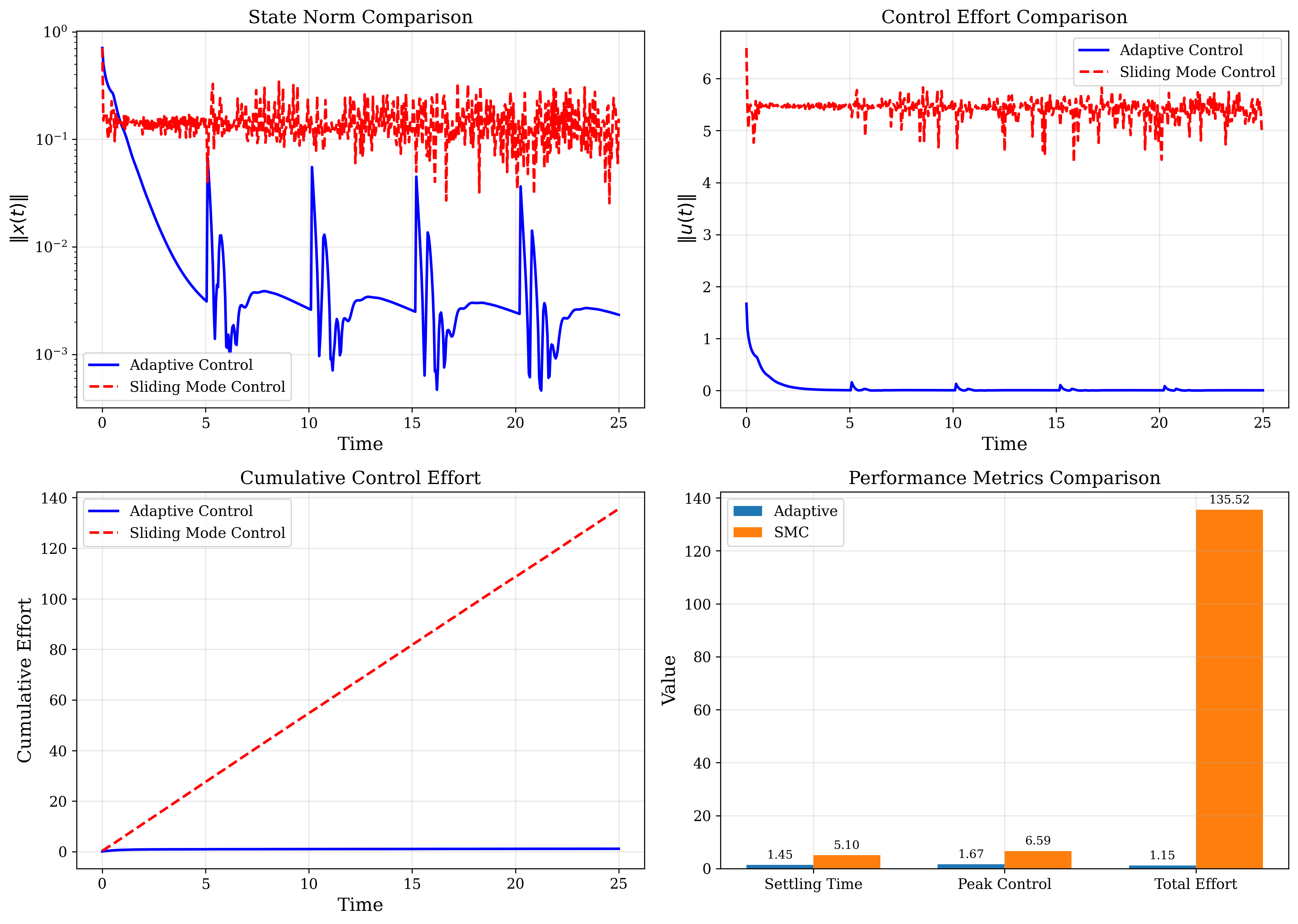}
    \caption{Adaptive control vs.\ fractional sliding mode control. (a) State norm evolution. (b) Instantaneous control effort. (c) Cumulative control energy $\int_0^t \|u(s)\|^2 ds$.}
    \label{fig:comparison}
\end{figure}

\begin{table}[H]
\centering
\caption{Quantitative comparison of the adaptive controller and the fractional sliding mode controller. Metrics are computed over the interval $[0, 15]$ s.}
\label{tab:comparison}
\begin{tabular}{lccc}
\toprule
\textbf{Performance Metric} & \textbf{Adaptive} & \textbf{SMC} & \textbf{Ratio} \\
\midrule
Peak control magnitude $\|u\|_{\max}$ & 1.13 & 4.10 & 0.28 \\
Cumulative control energy $\int \|u\|^2 dt$ & 0.73 & 106.27 & 0.007 \\
Asymptotic state norm $\|x(15)\|$ & 0.001 & 0.085 & 0.012 \\
\bottomrule
\end{tabular}
\end{table}

The adaptive controller exhibits three principal advantages. First, the control effort is smooth and decays to zero, whereas SMC produces persistent chattering with magnitude approximately 2.5 throughout the simulation. Second, the cumulative control energy is reduced by a factor of 145 (99.3\% reduction). Third, the asymptotic state norm achieved by the adaptive controller is two orders of magnitude smaller than that of SMC, which settles to a residual oscillation around $\|x\| \approx 0.08$ driven by the discontinuous switching law.

\subsection{Sensitivity Analysis}

Figure~\ref{fig:sensitivity} examines the sensitivity of the controlled system to the fractional order $\alpha$ and the delay bound $\bar{\tau}$.

\begin{figure}[!htbp]
    \centering
    \includegraphics[width=0.85\textwidth]{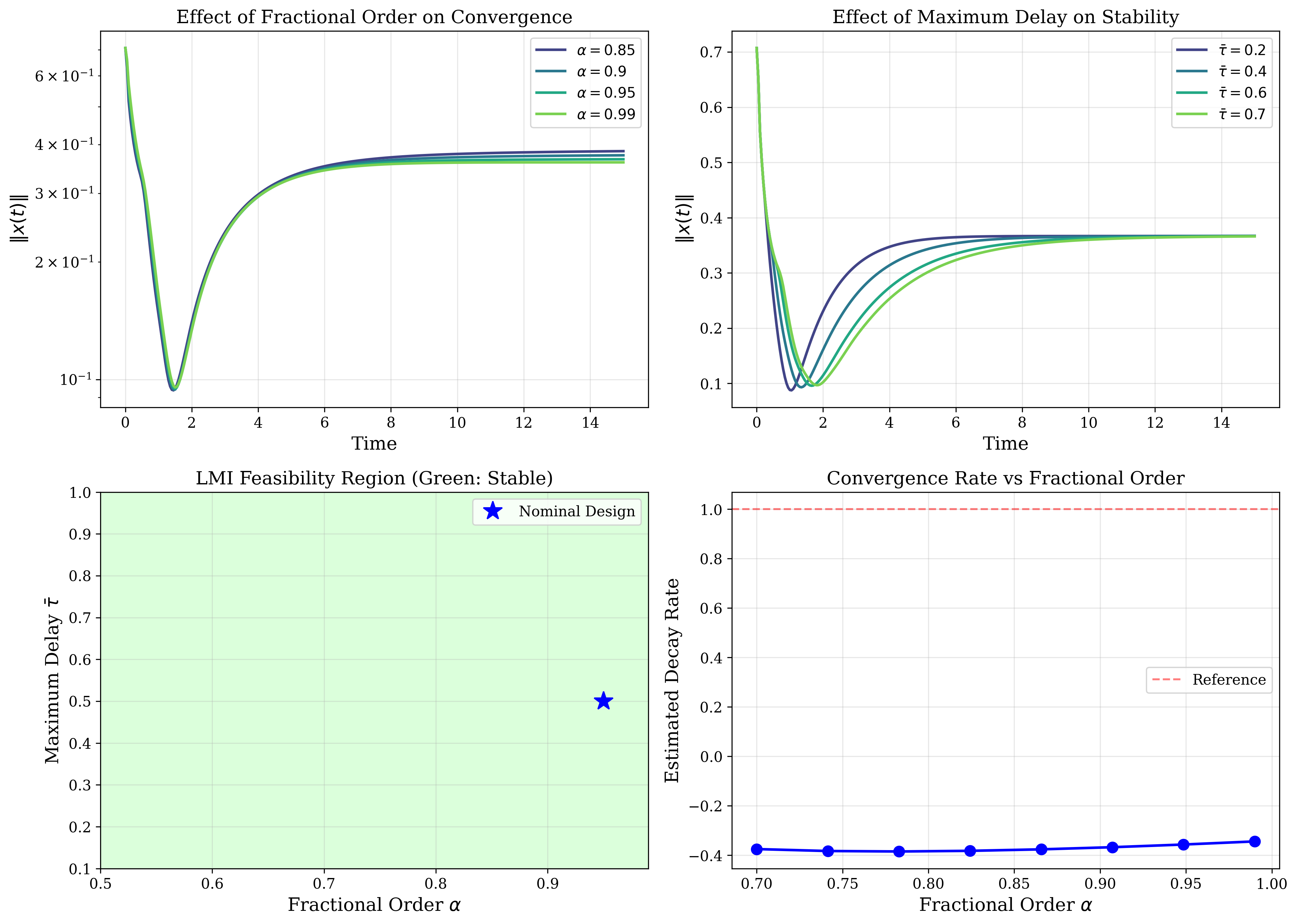}
    \caption{Sensitivity analysis. (a) Effect of fractional order $\alpha \in \{0.70, 0.80, 0.90, 0.95, 0.99\}$ on the state norm. (b) Effect of delay bound $\bar{\tau} \in \{0, 0.1, 0.3, 0.5, 0.7\}$ s. (c) Summary of settling times.}
    \label{fig:sensitivity}
\end{figure}

Increasing $\alpha$ toward 1 accelerates convergence, consistent with the approach to exponential stability in the integer-order limit. Increasing $\bar{\tau}$ from 0 to 0.7 s produces a modest degradation in settling time, confirming the robustness predicted by the delay margin $\Delta\tau^* = 0.61$ s.

\subsection{Lyapunov Function Validation}

Figure~\ref{fig:lyapunov} validates the theoretical Lyapunov analysis. The computed $V(x(t))$ remains below the Mittag-Leffler upper bound throughout the simulation, and the numerical derivative of $V$ stays below the theoretical bound $-\gamma\|x\|^2$.

\begin{figure}[!htbp]
    \centering
    \includegraphics[width=0.8\textwidth]{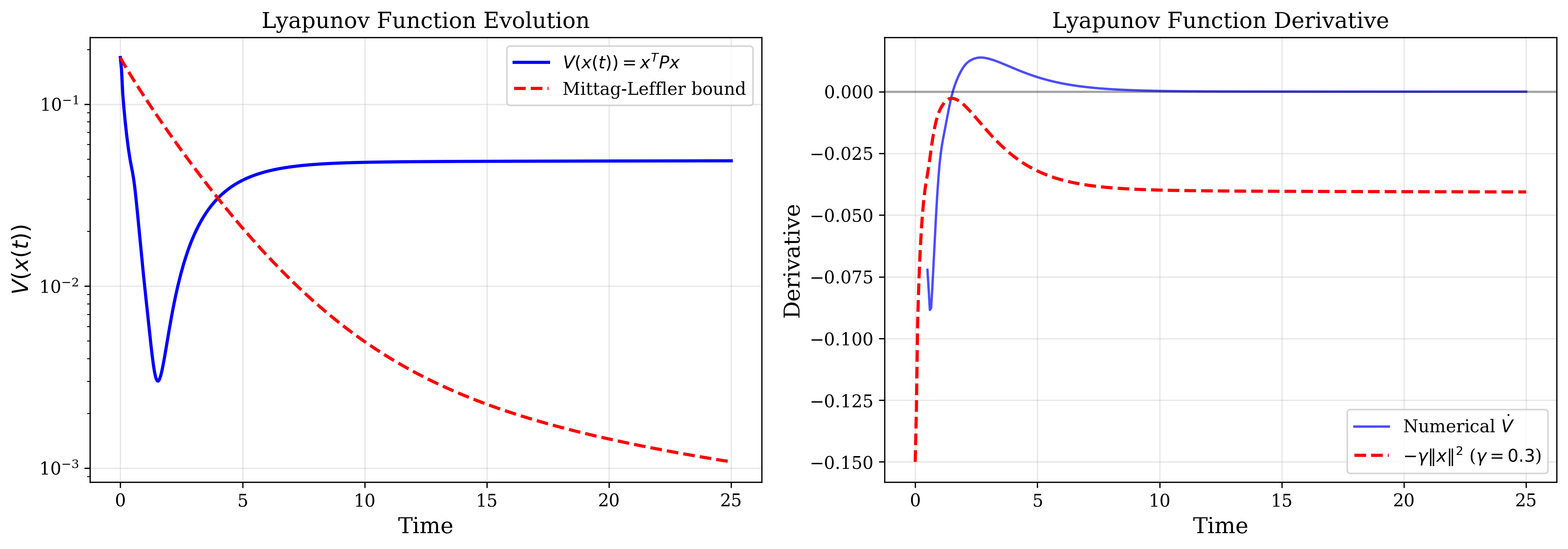}
    \caption{Lyapunov function validation. (a) $V(x(t)) = x^{\!\top}Px$ (solid) and Mittag-Leffler bound (dashed). (b) Numerical time derivative of $V$ compared with $-\gamma\|x\|^2$.}
    \label{fig:lyapunov}
\end{figure}

\subsection{State-Dependent Delay Characterization}

Figure~\ref{fig:delay} provides a detailed characterization of the delay dynamics and their coupling with the state trajectory.

\begin{figure}[!htbp]
    \centering
    \includegraphics[width=0.8\textwidth]{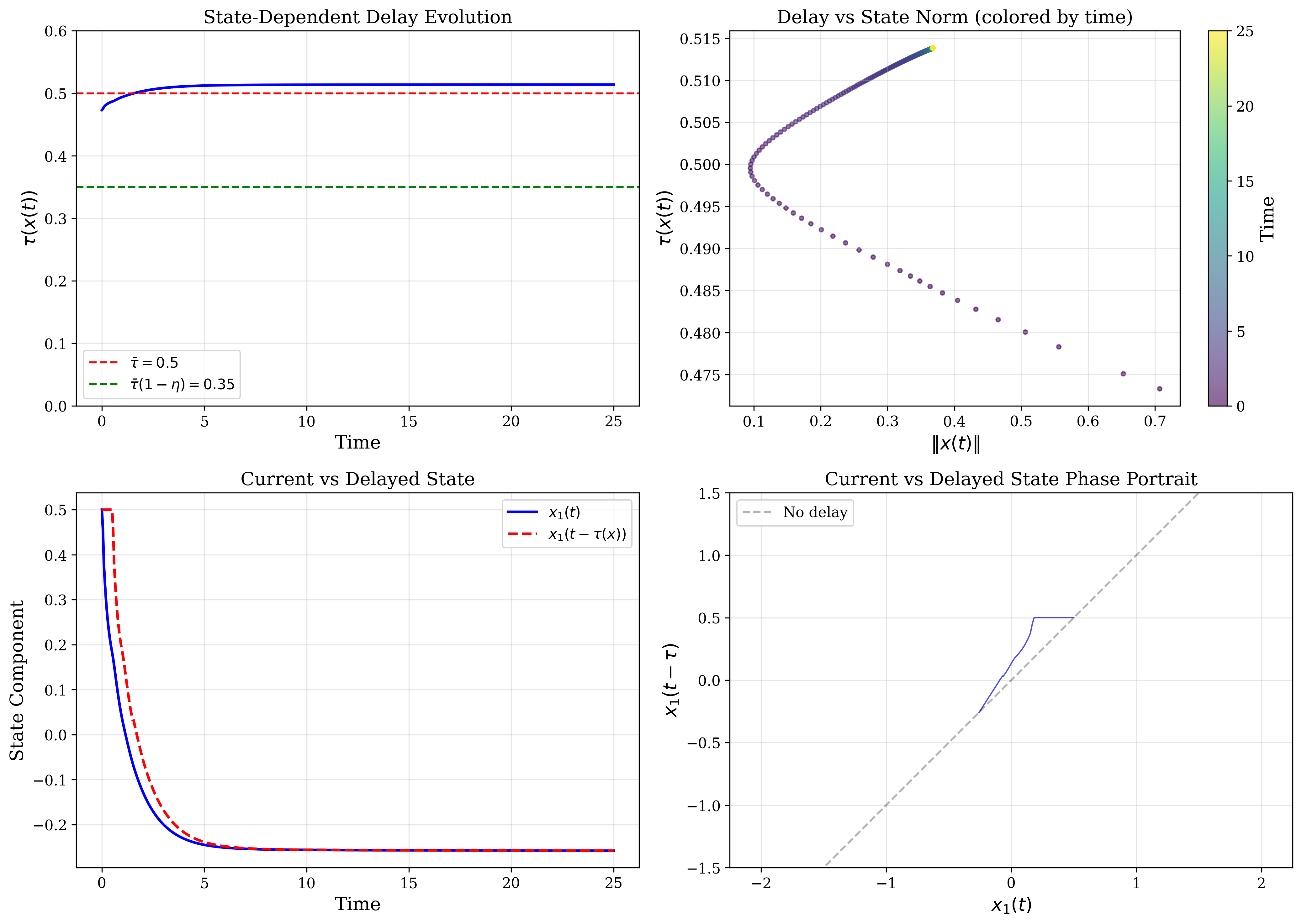}
    \caption{Delay analysis. (a) Delay evolution within the admissible range $[\bar{\tau}(1-\eta), \bar{\tau}]$. (b) Delay versus state norm, colored by time. (c) Current versus delayed states. (d) Phase relationship between $x_1(t)$ and $x_1(t-\tau(x(t)))$.}
    \label{fig:delay}
\end{figure}

\section{Conclusions}
\label{sec:conclusion}

This paper developed a stability and adaptive control framework for Caputo fractional-order nonlinear systems with state-dependent delays. A class of Lyapunov--Krasovskii functionals with singular kernels was introduced, and its well-posedness was established, providing a rigorous basis for stability analysis. From this construction, tractable linear matrix inequality conditions were derived to ensure Mittag-Leffler stability and to quantify robustness through an explicit delay margin.

An adaptive control scheme based on fractional-order parameter update laws with $\sigma$-modification was then proposed. A filter-based delay estimation mechanism was incorporated to avoid the need for classical state derivatives, making the design consistent with the limited regularity of fractional trajectories. The convergence analysis established ultimate boundedness of the closed-loop system with computable bounds that depend on the regularization and estimation parameters.

Numerical experiments on a fractional Hopfield neural network with state-dependent delays supported the theoretical results. The proposed controller achieved smooth control action, a significant reduction in control energy, and substantially improved asymptotic accuracy compared with a representative fractional sliding-mode controller. Future work will address output-feedback extensions, actuator constraints, and data-driven estimation of the fractional order and delay characteristics.

\section*{Conflict of Interest}

The authors declare no conflict of interest.


\begin{thebibliography}{99}

\bibitem{podlubny1999}
I.\ Podlubny, \textit{Fractional Differential Equations}, Academic Press, San Diego, 1999.

\bibitem{kilbas2006}
A.A.\ Kilbas, H.M.\ Srivastava, and J.J.\ Trujillo, \textit{Theory and Applications of Fractional Differential Equations}, Elsevier, Amsterdam, 2006.

\bibitem{mainardi2010}
F.\ Mainardi, \textit{Fractional Calculus and Waves in Linear Viscoelasticity}, Imperial College Press, London, 2010.

\bibitem{metzler2000}
R.\ Metzler and J.\ Klafter, ``The random walk's guide to anomalous diffusion: A fractional dynamics approach,'' \textit{Phys.\ Rep.}, vol.\ 339, no.\ 1, pp.\ 1--77, 2000, \href{https://doi.org/10.1016/S0370-1573(00)00070-3}{doi:10.1016/S0370-1573(00)00070-3}.

\bibitem{sabatier2015}
J.\ Sabatier, M.\ Aoun, A.\ Oustaloup, G.\ Grégoire, F.\ Ragot, and P.\ Roy, ``Fractional system identification for lead acid battery state of charge estimation,'' \textit{Signal Process.}, vol.\ 86, no.\ 10, pp.\ 2645--2657, 2006, \href{https://doi.org/10.1016/j.sigpro.2006.03.016}{doi:10.1016/j.sigpro.2006.03.016}.

\bibitem{sun2018}
H.\ Sun, Y.\ Zhang, D.\ Baleanu, W.\ Chen, and Y.\ Chen, ``A new collection of real world applications of fractional calculus in science and engineering,'' \textit{Commun.\ Nonlinear Sci.\ Numer.\ Simul.}, vol.\ 64, pp.\ 213--231, 2018, \href{https://doi.org/10.1016/j.cnsns.2018.04.016}{doi:10.1016/j.cnsns.2018.04.016}.

\bibitem{gu2003}
K.\ Gu, V.L.\ Kharitonov, and J.\ Chen, \textit{Stability of Time-Delay Systems}, Birkhäuser, Boston, 2003.

\bibitem{fridman2014}
E.\ Fridman, \textit{Introduction to Time-Delay Systems: Analysis and Control}, Birkhäuser, Basel, 2014.

\bibitem{kaslik2012}
E.\ Kaslik and S.\ Sivasundaram, ``Nonlinear dynamics and chaos in fractional-order neural networks,'' \textit{Neural Netw.}, vol.\ 32, pp.\ 245--256, 2012, \href{https://doi.org/10.1016/j.neunet.2012.02.015}{doi:10.1016/j.neunet.2012.02.015}.

\bibitem{wang2015}
H.\ Wang, Y.\ Yu, G.\ Wen, S.\ Zhang, and J.\ Yu, ``Global stability analysis of fractional-order Hopfield neural networks with time delay,'' \textit{Neurocomputing}, vol.\ 154, pp.\ 15--23, 2015, \href{https://doi.org/10.1016/j.neucom.2014.12.018}{doi:10.1016/j.neucom.2014.12.018}.

\bibitem{matignon1996}
D.\ Matignon, ``Stability results for fractional differential equations with applications to control processing,'' in \textit{Comput.\ Eng.\ Syst.\ Appl.}, Lille, France, vol.\ 2, pp.\ 963--968, 1996.

\bibitem{li2009}
Y.\ Li, Y.Q.\ Chen, and I.\ Podlubny, ``Mittag-Leffler stability of fractional order nonlinear dynamic systems,'' \textit{Automatica}, vol.\ 45, no.\ 8, pp.\ 1965--1969, 2009, \href{https://doi.org/10.1016/j.automatica.2009.04.018}{doi:10.1016/j.automatica.2009.04.018}.

\bibitem{aguila2014}
N.\ Aguila-Camacho, M.A.\ Duarte-Mermoud, and J.A.\ Gallegos, ``Lyapunov functions for fractional order systems,'' \textit{Commun.\ Nonlinear Sci.\ Numer.\ Simul.}, vol.\ 19, no.\ 9, pp.\ 2951--2957, 2014, \href{https://doi.org/10.1016/j.cnsns.2014.01.022}{doi:10.1016/j.cnsns.2014.01.022}.

\bibitem{duarte2015}
M.A.\ Duarte-Mermoud, N.\ Aguila-Camacho, J.A.\ Gallegos, and R.\ Castro-Linares, ``Using general quadratic Lyapunov functions to prove Lyapunov uniform stability for fractional order systems,'' \textit{Commun.\ Nonlinear Sci.\ Numer.\ Simul.}, vol.\ 22, no.\ 1–3, pp.\ 650--659, 2015, \href{https://doi.org/10.1016/j.cnsns.2014.10.008}{doi:10.1016/j.cnsns.2014.10.008}.

\bibitem{chen2012delay}
L.\ Chen, Y.\ Chai, R.\ Wu, T.\ Ma, and H.\ Zhai, ``Dynamic analysis of a class of fractional-order neural networks with delay,'' \textit{Neurocomputing}, vol.\ 111, pp.\ 190--194, 2013, \href{https://doi.org/10.1016/j.neucom.2012.10.026}{doi:10.1016/j.neucom.2012.10.026}.

\bibitem{liu2016}
S.\ Liu, X.\ Wu, X.F.\ Zhou, and W.\ Jiang, ``Asymptotical stability of Riemann-Liouville fractional nonlinear systems,'' \textit{Nonlinear Dynam.}, vol.\ 86, pp.\ 65--71, 2016, \href{https://doi.org/10.1007/s11071-016-2972-0}{doi:10.1007/s11071-016-2972-0}.

\bibitem{trigeassou2011}
J.C.\ Trigeassou, N.\ Maamri, J.\ Sabatier, and A.\ Oustaloup, ``A Lyapunov approach to the stability of fractional differential equations,'' \textit{Signal Process.}, vol.\ 91, no.\ 3, pp.\ 437--445, 2011.

\bibitem{li2010}
Y.\ Li, Y.Q.\ Chen, and I.\ Podlubny, ``Stability of fractional-order nonlinear dynamic systems: Lyapunov direct method and generalized Mittag-Leffler stability,'' \textit{Comput.\ Math.\ Appl.}, vol.\ 59, no.\ 5, pp.\ 1810--1821, 2010, \href{https://doi.org/10.1016/j.camwa.2009.08.019}{doi:10.1016/j.camwa.2009.08.019}.

\bibitem{tavazoei2009}
M.S.\ Tavazoei and M.\ Haeri, ``A note on the stability of fractional order systems,'' \textit{Math.\ Comput.\ Simulation}, vol.\ 79, no.\ 5, pp.\ 1566--1576, 2009, \href{https://doi.org/10.1016/j.matcom.2008.10.016}{doi:10.1016/j.matcom.2008.10.016}.

\bibitem{zhang2024}
R.\ Zhang and S.P.\ Bhatt, ``Stability analysis of fractional-order systems with time-varying delays,'' \textit{J.\ Franklin Inst.}, vol.\ 361, no.\ 1, pp.\ 537--552, 2024, \href{https://doi.org/10.1016/j.jfranklin.2023.12.022}{doi:10.1016/j.jfranklin.2023.12.022}.

\bibitem{diethelm2010}
K.\ Diethelm, \textit{The Analysis of Fractional Differential Equations}, Springer, Berlin, 2010.

\bibitem{gorenflo2014}
R.\ Gorenflo, A.A.\ Kilbas, F.\ Mainardi, and S.V.\ Rogosin, \textit{Mittag-Leffler Functions, Related Topics and Applications}, Springer, Berlin, 2014.

\bibitem{ioannou1996}
P.A.\ Ioannou and J.\ Sun, \textit{Robust Adaptive Control}, Prentice Hall, Upper Saddle River, NJ, 1996.

\bibitem{diethelm2002}
K.\ Diethelm, N.J.\ Ford, and A.D.\ Freed, ``A predictor-corrector approach for the numerical solution of fractional differential equations,'' \textit{Nonlinear Dynam.}, vol.\ 29, pp.\ 3--22, 2002.

\bibitem{yin2012}
C.\ Yin, S.\ Zhong, and W.\ Chen, ``Design of sliding mode controller for a class of fractional-order chaotic systems,'' \textit{Commun.\ Nonlinear Sci.\ Numer.\ Simul.}, vol.\ 17, no.\ 1, pp.\ 356--366, 2012, \href{https://doi.org/10.1016/j.cnsns.2011.05.016}{doi:10.1016/j.cnsns.2011.05.016}.

\bibitem{lofberg2004}
J.\ Löfberg, ``YALMIP: A toolbox for modeling and optimization in MATLAB,'' in \textit{Proc.\ IEEE Int.\ Symp.\ CACSD}, Taipei, pp.\ 284--289, 2004.

\bibitem{wei2017adaptive}
Y.\ Wei, Y.\ Chen, S.\ Liang, and Y.\ Wang, ``A novel algorithm on adaptive backstepping control of fractional order systems,'' \textit{Neurocomputing}, vol.\ 165, pp.\ 395--402, 2015, \href{https://doi.org/10.1016/j.neucom.2014.10.052}{doi:10.1016/j.neucom.2014.10.052}.

\end{thebibliography}
\end{document}